\documentclass{amsart}
\usepackage{latexsym,amsxtra,amscd,ifthen}
\usepackage{amsfonts}
\usepackage{verbatim}
\usepackage{amsmath}
\usepackage{amsthm}
\usepackage{amssymb}
\usepackage{color}
\usepackage{newtxmath}
\usepackage{bbm}
\usepackage{nicematrix}
\usepackage{xcolor}
\definecolor{navy}{HTML}{004d99}
\usepackage[matrix,arrow]{xy}
\usepackage{hyperref}
\hypersetup{
 colorlinks  = true,
 urlcolor   = navy,
 linkcolor  = navy,
 citecolor  = black
}

\allowdisplaybreaks
\numberwithin{equation}{section}

\theoremstyle{plain}
\newtheorem{theorem}{Theorem}[section]
\newtheorem{lemma}[theorem]{Lemma}

\allowdisplaybreaks

\theoremstyle{definition}
\newtheorem{definition}[theorem]{Definition}

\makeatletter       
\let\c@equation\c@theorem 
\makeatother

\allowdisplaybreaks

\begin{document}

\title[ The Rota-Baxter algebra structures on  split semi-quaternion algebra]
{ The Rota-Baxter algebra structures on  split semi-quaternion algebra}

\author[Chen]{Quanguo Chen}
\address{(Chen)College of Mathematics and Statistics, Kashi University, Xinjiang, Kashi, China}
\email{cqg211@163.com}

\author[Deng]{Yong Deng}
\address{(Deng)College of Mathematics and Statistics, Kashi University, Xinjiang, Kashi, China}
\email{dengyongks@126.com}


\begin{abstract}
 In this paper, we shall describe all the Rota-Baxter operators with any weight on split semi-quaternion algebra. Firstly, we give the matrix characterization of the Rota-Baxter operator on split semi-quaternion algebra. Then we give the corresponding matrix representations of all the Rota-Baxter operators with any weight on split semi-quaternion algebra. Finally, we shall prove that the Ma et al. results  about the Rota-Baxter operators on  Sweedler algebra are just special cases of our results. 
\end{abstract}

\thanks{This work was supported by the National Natural Science
	Foundation of China (No. 12271292) and the Natural Foundation of Shandong Province(No. ZR2022MA002).
}
\maketitle

\setcounter{section}{0}
\section{Introduction}
\label{xxsec0}

In the 1960s, in order to solve the analytical and combinatorial problems, the Rota-Baxter operators are introduced and generalized(\cite{Baxter},\cite{Rota1},\cite{Rota2}). The Rota-Baxter operator can be regarded as an algebraic abstraction of  the integral operator.   For more than half a century, the Rota-Baxter operator theory is progressively applied to the quantum field theory(\cite{Connes1},\cite{Connes2},\cite{Connes3}), Operads(\cite{Leroux1},\cite{Leroux2}), Hopf algebras(\cite{Ebrahimi2}),  commutative algebra, combinatorics and number theory(\cite{Guo1},\cite{Ebrahimi1},
\cite{Guo2}). The broad connections of Rota-Baxter operator with many areas in mathematics and mathematical physics are remarkable.  More about the Rota-Baxter operator theory and application can be seen in\cite{Jian1} and \cite{Jian2}.

Quaternions were invented by William Rowan Hamilton as an extension to the complex number in 1843.  Quaternions are extensively used in
many areas such as computer science, physics, differential geometry, quantum
physics and pure algebra (\cite{Rodman}, \cite{Vince}). As  new quaternions, a brief  introduction of the split semi-quaternions is provided in \cite{Rosenfeld}, and some of their algebraic properties are given in  \cite{Jafari}.

One of the interesting direction in the study of Rota–Baxter operators is a problem of classification of Rota–Baxter operators on a given algebra.  The Rota-Baxter operators on 2 and 3 dimensional algebra are studied in \cite{An},\cite{Guo6} and \cite{Li}. The Rota-Baxter operator of the second-order full-matrix algebraic weight 0 on complex fields by both standard and Grobner basis methods was given in \cite{Tang}.  The Rota-Baxter operators on four-dimensional Sweedler algebra considered as an associative algebra were found in \cite{Ma} and \cite{Ma1}.

The purpose of this paper is to characterize all the Rota-Baxter operators on split semi-quaternion algebra.

The paper is organized as follows.

In Sec.2, we recall some basic definitions and results for split  semi-quaternion algebra and the  Rota-Baxter operator. We give the matrix characterization of the Rota-Baxter operator on split semi-quaternion algebra in Sec.3.
We shall describe  all the Rota-Baxter operators with weight 0    in Sec.4. In Sec. 5, all  the Rota-Baxter operators with non-zero weight are given. We shall discuss the relationships between  Ma et al. results  about the Rota-Baxter operators on  Sweedler algebra and our results in Sec. 6.

\section{preliminaries}
\label{xxsec1}
Throughout the paper, $\mathbb{R}$ denotes the real number field. All algebras
are over $\mathbb{R}$ and linear means $\mathbb{R}$-linear.   Given a matrix $M$, $M^{T}$ means the transpose of $M$, and 
$$M(4)=\left(
\begin{array}{cccc}
	M & 0 & 0 & 0 \\
	0 & M & 0 & 0 \\
	0 & 0 & M & 0 \\
	0 & 0 & 0 & M \\
\end{array}
\right).$$

\subsection{Split Semi-quaternion algebra}

A split semi-quaternion $q$ is an expression of
the form
$q=a_{0}e_{0}+a_{1}e_{1}+a_{2}e_{2}+a_{3}e_{3},$
where $a_{0}, a_{1}, a_{2}, a_{3}$ are real numbers and $e_{0}=1,e_{1}, e_{2}, e_{3}$ satisfy the following equalities:
$$
 e_{1}^{2}=1, e_{2}^{2}=0, e_{3}^{2}=0, e_{1}e_{2}=e_{3}=-e_{2}e_{1}, e_{2}e_{3}=0=e_{3}e_{2}, e_{3}e_{1}=-e_{2}=-e_{1}e_{3}.
$$
The set of all  split semi-quaternion is denoted by $\mathbb{H}_{ss}$. The addition
and multiplication of $\mathbb{H}_{ss}$ are given as follows: for any $q=a_{0}+a_{1}e_{1}+a_{2}e_{2}+a_{3}e_{3}, p=b_{0}+b_{1}e_{1}+b_{2}e_{2}+b_{3}e_{3}\in \mathbb{H}_{s}$,
$$
q+p=(a_{0}+b_{0})+(a_{1}+b_{1})e_{1}+(a_{2}+b_{2})e_{2}+(a_{3}+b_{3})e_{3},
$$
\begin{eqnarray*}
	qp&=&(a_{0}b_{0}+a_{1} b_{1})
	+(a_{1} b_{0}+a_{0} b_{1})e_{1}\\
	& &+(a_{2} b_{0}-a_{3} b_{1}+ a_{0} b_{2}+a_{1} b_{3})e_{2}
	+(a_{3} b_{0}- a_{2} b_{1}+ a_{1} b_{2}+ a_{0} b_{3})e_{3}.
\end{eqnarray*}
It is easily checked that  $\mathbb{H}_{ss}$  with the above addition and multiplication is an associative
algebra. We call  $\mathbb{H}_{ss}$ an algebra of split semi-quaternions (or split semi-quaternion algebra).

\subsection{ Rota-Baxter operators}
\begin{definition}
	Let $\lambda$ be a given element of $\mathbb{R}$. A Rota-Baxter algebra of weight $\lambda$ is a pair $(A, \mathcal{P})$ consisting of an algebra $A$ and a linear operator $\mathcal{P}: A\rightarrow A$ that satisfies the Rota-Baxter equation
	\begin{equation}\label{c1}\tag{2.1}
		\mathcal{P}(x)\mathcal{P}(y)=\mathcal{P}(\mathcal{P}(x)y)+\mathcal{P}(x\mathcal{P}(y))+\lambda \mathcal{P}(xy), \forall x,y \in \mathbb{C}.
	\end{equation}
	Then $\mathcal{P}$ is called a Rota-Baxter operator of weight $\lambda$.
\end{definition}

\section{Rota-Baxter operators on $\mathbb{H}_{ss}$}

Let $\mathcal{P}$ be a linear transformation on $\mathbb{H}_{ss}$, and $$
P=\left(
\begin{array}{cccc}
	a_{11} & a_{12}& a_{13} & a_{14}\\
	a_{21} & a_{22}& a_{23} & a_{24} \\
	a_{31} & a_{32}& a_{33} & a_{34} \\
	a_{41} & a_{42}& a_{43} & a_{44} \\
\end{array}
\right)=(\gamma_{0},\gamma_{1},\gamma_{2},\gamma_{3})
$$ the matrix of $\mathcal{P}$ with  respect to the basis $e_{0},e_{1}, e_{2}, e_{3}$. Notice easily that $\gamma_{i}$ is the coordinate of  $\mathcal{P}$ with  respect to the basis $e_{0},e_{1}, e_{2}, e_{3}$. Thus we have
$$
\mathcal{P}(e_{i})=(e_{0},e_{1}, e_{2}, e_{3})\gamma_{i}=\gamma_{i}^{T}\left(
\begin{array}{c}
	e_{0} \\
	e_{1} \\
	e_{2} \\
	e_{3} \\
\end{array}
\right).
$$

\begin{lemma}\label{lem1}
	Let $\mathcal{P}$ be a linear transformation on $\mathbb{H}_{ss}$. The following equation   holds:
	\begin{equation}\label{t11}\tag{3.1}
		\mathcal{P}(e_{i})\mathcal{P}(e_{j})
		=\gamma_{i} ^{T}C\gamma_{j}(4)
		\left(
		\begin{array}{c}
			e_{0}\\
			e_{1} \\
			e_{2} \\
			e_{3} \\
		\end{array}
		\right),
	\end{equation}
	where
	$$
	C=
	\left(
	\begin{array}{cccccccccccccccc}
		1 & 0 & 0 & 0 & 0 & 1 & 0 & 0 & 0 & 0 & 1 & 0& 0 & 0 & 0 & 1 \\
		0 & 1& 0 & 0 & 1 & 0 & 0 & 0 & 0 & 0 & 0 & 1& 0 & 0 & 1 & 0 \\
		0 & 0 & 0& 0 & 0 & 0& 0 & 0 & 1& 0 & 0 & 0& 0 & -1 & 0 & 0\\
		0& 0 & 0 & 0 & 0 & 0 & 0& 0 & 0 & -1 & 0 & 0& 1 & 0 & 0 & 0 \\
	\end{array}
	\right).
	$$

\end{lemma}
\begin{proof}
	For any $i,j$, since \begin{eqnarray*}
		\mathcal{P}(e_{i})\mathcal{P}(e_{j})
		&=&\gamma_{i}^{T}
		\left(
		\begin{array}{c}
			e_{0} \\
			e_{1} \\
			e_{2} \\
			e_{3} \\
		\end{array}
		\right)
		(e_{0},e_{1},e_{2},e_{3})\gamma_{j}\\
		&=&\gamma_{i}^{T}
		\left(
		\begin{array}{cccc}
			e_{0} & e_{1}& e_{2} & e_{3} \\
			e_{1} &  e_{0} & e_{3} &  e_{2}\\
			e_{2} & -e_{3} & 0 & 0 \\
			e_{3} & -e_{2} & 0 & 0 \\
		\end{array}
		\right)\gamma_{j}\\
		&=&\gamma_{i}^{T}
		\left(
		\begin{array}{cccc}
			1 & 0& 0 & 0 \\
			0 & 1  & 0 & 0\\
			0 & 0 & 0 & 0 \\
			0 & 0 & 0 & 0 \\
		\end{array}
		\right)\gamma_{j}e_{0}
		+\gamma_{i}^{T}
		\left(
		\begin{array}{cccc}
			0& 1& 0 & 0 \\
			1 & 0 & 0 & 0\\
			0 & 0 & 0 & 0 \\
			0 & 0 & 0& 0  \\
		\end{array}
		\right)\gamma_{j}e_{1}\\
		& &+\gamma_{i}^{T}
		\left(
		\begin{array}{cccc}
			0 & 0& 1 & 0 \\
			0 & 0 & 0 & 0\\
			1 & 0 & 0  & 0 \\
			0 & 0 & 0 & 0  \\
		\end{array}
		\right)\gamma_{j}e_{2}
		+\gamma_{i}^{T}
		\left(
		\begin{array}{cccc}
			0& 0& 0 & 1 \\
			0 & 0 & 1 & 0\\
			0 & -1 & 0 & 0 \\
			1 & 0 & 0 & 0  \\
		\end{array}
		\right)\gamma_{j}e_{3}\\
		&=&\gamma_{i} ^{T}C\left(
		\begin{array}{cccc}
			\gamma_{j}&  &  &  \\
			& \gamma_{j} &  &  \\
			& & \gamma_{j}&  \\
			&  &  &\gamma_{j}\\
		\end{array}
		\right)
		\left(
		\begin{array}{c}
			e_{0}\\
			e_{1} \\
			e_{2} \\
			e_{3} \\
		\end{array}
		\right),
	\end{eqnarray*}
	so this yields the equation (\ref{t11}). Other equations can be similarly checked.
\end{proof}
\begin{lemma}\label{lem2}
	Let $\mathcal{P}$ be a linear transformation on $\mathbb{H}_{ss}$. The following equations   hold:
		$$
	e_{i}\mathcal{P}(e_{j})=\gamma_{j}^{T}E_{i}^{T}\left(
	\begin{array}{c}
		e_{0}\\
		e_{1} \\
		e_{2} \\
		e_{3} \\
	\end{array}
	\right),
	\mathcal{P}(e_{i})e_{1}=\gamma_{i}^{T}
	\left(
	\begin{array}{cccc}
		0 & 1 & 0 & 0 \\
		1 & 0 & 0 & 0 \\
		0 & 0 & 0 & -1 \\
		0 & 0 & -1 & 0 \\
	\end{array}
	\right)\left(
	\begin{array}{c}
		e_{0}\\
		e_{1} \\
		e_{2} \\
		e_{3} \\
	\end{array}
	\right),
	$$
	$$
	\mathcal{P}(E_{i})e_{2}=\gamma_{i}^{T}
	\left(
	\begin{array}{cccc}
		0 & 0 & 1 & 0 \\
		0 & 0 & 0 & 1 \\
		0& 0 & 0 & 0 \\
		0 & 0& 0 & 0 \\
	\end{array}
	\right)\left(
	\begin{array}{c}
		e_{0}\\
		e_{1} \\
		e_{2} \\
		e_{3} \\
	\end{array}
	\right),
	\mathcal{P}(e_{i})e_{3}=\gamma_{i}^{T}
	\left(
	\begin{array}{cccc}
		0 & 0 & 0 & 1 \\
		0 & 0 & 1 & 0 \\
		0 & 0 & 0 & 0 \\
		0 & 0 & 0 & 0 \\
	\end{array}
	\right)\left(
	\begin{array}{c}
		e_{0}\\
		e_{1} \\
		e_{2} \\
		e_{3} \\
	\end{array}
	\right),
	$$
where 
$$
E_{0}=\left(
\begin{array}{cccc}
	1 & 0 & 0 & 0 \\
	0 & 1 & 0 & 0 \\
	0 & 0 & 1 & 0 \\
	0 & 0 & 0 & 1 \\
\end{array}
\right), E_{1}=\left(
\begin{array}{cccc}
	0 & 1 & 0 & 0 \\
	1 & 0 & 0 & 0 \\
	0 & 0 & 0 & 1 \\
	0 & 0 & 1 & 0 \\
\end{array}
\right),E_{2}=\left(
\begin{array}{cccc}
	0 & 0 & 0 & 0 \\
	0 & 0 & 0 & 0 \\
	1 & 0 & 0 & 0 \\
	0 & -1 & 0 & 0 \\
\end{array}
\right),
$$$$E_{3}=\left(
\begin{array}{cccc}
	0 & 0 & 0 & 0 \\
	0 & 0 & 0 & 0 \\
	0 & -1 & 0 & 0 \\
	1 & 0 & 0 & 0 \\
\end{array}
\right).
$$
\end{lemma}
\begin{proof}
	Straightforward. 
\end{proof}

From lemmas \ref{lem1} and \ref{lem2}, the following main theorem can be obtained.

\begin{theorem}\label{thm1}
	$\mathcal{P}$ is a Rota-Baxter operator of weight $\lambda$ on $\mathbb{H}_{ss}$ if and only if the column vectors  $\gamma_{i}$ of $P$  and $P$ satisfy the following equation:
	$$
	\gamma_{0}(4)^{T}C^{T}P=P(E_{0},E_{1},E_{2},E_{3})\gamma_{0}(4)+P^{2}+\lambda P,
	$$
	\begin{eqnarray*}
	\gamma_{1}(4)^{T}C^{T}P=P(E_{0},E_{1},E_{2},E_{3})\gamma_{1}(4)+P\left(
		\begin{array}{cccc}
			0 & 1 & 0 & 0 \\
			1 & 0 & 0 & 0 \\
			0 & 0 & 0 & -1 \\
			0 & 0 & -1 & 0 \\
		\end{array}
		\right)^{T}P+\lambda P\left(
		\begin{array}{cccc}
			0 & 1 & 0 & 0 \\
			1 & 0 & 0 & 0 \\
			0 & 0 & 0 & -1 \\
			0 & 0 & -1 & 0 \\
		\end{array}
		\right)^{T},
	\end{eqnarray*}
	
	\begin{eqnarray*}
	\gamma_{2}(4)^{T}C^{T}P=P(E_{0},E_{1},E_{2},E_{3})\gamma_{2}(4)+P\left(
		\begin{array}{cccc}
			0 & 0 & 1 & 0 \\
			0 & 0 & 0 & 1 \\
			0& 0 & 0 & 0 \\
			0 & 0& 0 & 0 \\
		\end{array}
		\right)^{T}P+\lambda P\left(
		\begin{array}{cccc}
			0 & 0 & 1 & 0 \\
			0 & 0 & 0 & 1 \\
			0& 0 & 0 & 0 \\
			0 & 0& 0 & 0 \\
		\end{array}
		\right)^{T},
	\end{eqnarray*}
	\begin{eqnarray*}
	\gamma_{3}(4)^{T}C^{T}P=P(E_{0},E_{1},E_{2},E_{3})\gamma_{3}(4)+P\left(
		\begin{array}{cccc}
			0 & 0 & 0 & 1 \\
			0 & 0 & 1 & 0 \\
			0 & 0 & 0 & 0 \\
			0& 0 & 0 & 0 \\
		\end{array}
		\right)^{T}P+\lambda P\left(
		\begin{array}{cccc}
			0 & 0 & 0 & 1 \\
			0 & 0 &1 & 0 \\
			0 & 0 & 0 & 0 \\
			0 & 0 & 0 & 0 \\
		\end{array}
		\right)^{T},
	\end{eqnarray*}
	where $C$ and $E_{i}(i=0,1,2,3)$ are shown in Lemmas \ref{lem1} and \ref{lem2}.
\end{theorem}

\begin{proof}
	
	Using Lemma \ref{lem2}, we have
	$$
	\mathcal{P}(e_{i}\mathcal{P}(e_{j}))=\gamma_{j}^{T}E_{i}^{T}P^{T}\left(
	\begin{array}{c}
		e_{0}\\
		e_{1} \\
		e_{2} \\
		e_{3} \\
	\end{array}
	\right),
	$$
	$$
	\mathcal{P}(\mathcal{P}(e_{i})e_{1})=\gamma_{i}^{T}
	\left(
	\begin{array}{cccc}
		0 & 1 & 0 & 0 \\
		1 & 0 & 0 & 0 \\
		0 & 0 & 0 & -1 \\
		0 & 0 & -1 & 0 \\
	\end{array}
	\right)P^{T}\left(
	\begin{array}{c}
		e_{0}\\
		e_{1} \\
		e_{2} \\
		e_{3} \\
	\end{array}
	\right),
	$$
	$$
	\mathcal{P}(\mathcal{P}(e_{i})e_{2})=\gamma_{i}^{T}
	\left(
	\begin{array}{cccc}
		0 & 0 & 1 & 0 \\
		0 & 0 & 0 & 1 \\
		0& 0 & 0 & 0 \\
		0 & 0& 0 & 0 \\
	\end{array}
	\right)P^{T}\left(
	\begin{array}{c}
		e_{0}\\
		e_{1} \\
		e_{2} \\
		e_{3} \\
	\end{array}
	\right),
	$$
	$$
	\mathcal{P}(\mathcal{P}(e_{i})e_{3})=\gamma_{i}^{T}
	\left(
	\begin{array}{cccc}
		0 & 0 & 0 & 1 \\
		0 & 0 & 1 & 0 \\
		0 & 0 & 0 & 0 \\
		0 & 0 & 0 & 0 \\
	\end{array}
	\right)P^{T}\left(
	\begin{array}{c}
		e_{0}\\
		e_{1} \\
		e_{2} \\
		e_{3} \\
	\end{array}
	\right).
	$$
	$\mathcal{P}$ is a Rota-Baxter operator with weight $\lambda$ if and only if  for any $i, j = 0,1,2,3,$  the following equation holds:
	$$
	\mathcal{P}(e_{i})\mathcal{P}(e_{j})=\mathcal{P}(e_{i}\mathcal{P}(e_{j}))+\mathcal{P}(\mathcal{P}(e_{i})e_{j})+\lambda \mathcal{P}(e_{i}e_{j}).
	$$
	
	When $j=0$, from Lemma \ref{lem1}, one has
	$$
	\gamma_{i}^{T}C\left(
	\begin{array}{cccc}
		\gamma_{0}&  &  &  \\
		& \gamma_{0} &  &  \\
		& & \gamma_{0}&  \\
		&  &  &\gamma_{0}\\
	\end{array}
	\right)=\gamma_{0}^{T}E_{i}^{T}P^{T}+\gamma_{i}^{T}P^{T}+\lambda\gamma_{i}^{T},
	$$
	which is equivalent to the first equality in Theorem \ref{thm1}.
	
	When $j=1$,
	$$
	\gamma_{0}^{T}C\left(
	\begin{array}{cccc}
		\gamma_{1}&  &  &  \\
		& \gamma_{1} &  &  \\
		& & \gamma_{1}&  \\
		&  &  &\gamma_{1}\\
	\end{array}
	\right)=\gamma_{1}^{T}E_{0}P^{T}+\gamma_{0}^{T}\left(
	\begin{array}{cccc}
		0 & 1 & 0 & 0 \\
		1 & 0 & 0 & 0 \\
		0 & 0 & 0 & -1 \\
		0 & 0 & -1 & 0 \\
	\end{array}
	\right)P^{T}+\lambda\gamma_{1}^{T},
	$$
	$$
	\gamma_{1}^{T}C\left(
	\begin{array}{cccc}
		\gamma_{1}&  &  &  \\
		& \gamma_{1} &  &  \\
		& & \gamma_{1}&  \\
		&  &  &\gamma_{1}\\
	\end{array}
	\right)=\gamma_{1}^{T}E_{1}^{T}P^{T}+\gamma_{1}^{T}\left(
	\begin{array}{cccc}
		0 & 1 & 0 & 0 \\
		1 & 0 & 0 & 0 \\
		0 & 0 & 0 & -1 \\
		0 & 0 & -1 & 0 \\
	\end{array}
	\right)P^{T}+\lambda\gamma_{0}^{T},
	$$
	$$
	\gamma_{2}^{T}C\left(
	\begin{array}{cccc}
		\gamma_{1}&  &  &  \\
		& \gamma_{1} &  &  \\
		& & \gamma_{1}&  \\
		&  &  &\gamma_{1}\\
	\end{array}
	\right)=\gamma_{1}^{T}E_{2}^{T}P^{T}+\gamma_{2}^{T}\left(
	\begin{array}{cccc}
		0 & 1 & 0 & 0 \\
		1 & 0 & 0 & 0 \\
		0 & 0 & 0 & -1 \\
		0 & 0 & -1 & 0 \\
	\end{array}
	\right)P^{T}-\lambda\gamma_{3}^{T},
	$$
	$$
	\gamma_{3}^{T}C\left(
	\begin{array}{cccc}
		\gamma_{1}&  &  &  \\
		& \gamma_{1} &  &  \\
		& & \gamma_{1}&  \\
		&  &  &\gamma_{1}\\
	\end{array}
	\right)=\gamma_{1}^{T}E_{3}^{T}P^{T}+\gamma_{3}^{T}\left(
	\begin{array}{cccc}
		0 & 1 & 0 & 0 \\
		1 & 0 & 0 & 0 \\
		0 & 0 & 0 & -1 \\
		0 & 0 & -1 & 0 \\
	\end{array}
	\right)P^{T}-\lambda\gamma_{2}^{T},
	$$
	which is equivalent to the second equality in Theorem \ref{thm1}.
	
	When $j=2$,
	$$
	\gamma_{0}^{T}C\left(
	\begin{array}{cccc}
		\gamma_{2}&  &  &  \\
		& \gamma_{2} &  &  \\
		& & \gamma_{2}&  \\
		&  &  &\gamma_{2}\\
	\end{array}
	\right)=\gamma_{2}^{T}E_{0}P^{T}+\gamma_{0}^{T}\left(
	\begin{array}{cccc}
		0 & 0 & 1 & 0 \\
		0 & 0 & 0 & 1 \\
		0& 0 & 0 & 0 \\
		0 & 0& 0 & 0 \\
	\end{array}
	\right)P^{T}+\lambda\gamma_{2}^{T},
	$$
	$$
	\gamma_{1}^{T}C\left(
	\begin{array}{cccc}
		\gamma_{2}&  &  &  \\
		& \gamma_{2} &  &  \\
		& & \gamma_{2}&  \\
		&  &  &\gamma_{2}\\
	\end{array}
	\right)=\gamma_{2}^{T}E_{1}^{T}P^{T}+\gamma_{1}^{T}\left(
	\begin{array}{cccc}
		0 & 0 & 1 & 0 \\
		0 & 0 & 0 & 1 \\
		0& 0 & 0 & 0 \\
		0 & 0& 0 & 0 \\
	\end{array}
	\right)P^{T}+\lambda\gamma_{3}^{T},
	$$
	$$
	\gamma_{2}^{T}C\left(
	\begin{array}{cccc}
		\gamma_{2}&  &  &  \\
		& \gamma_{2} &  &  \\
		& & \gamma_{2}&  \\
		&  &  &\gamma_{2}\\
	\end{array}
	\right)=\gamma_{2}^{T}E_{2}^{T}P^{T}+\gamma_{2}^{T}\left(
	\begin{array}{cccc}
		0 & 0 & 1 & 0 \\
		0 & 0 & 0 & 1 \\
		0& 0 & 0 & 0 \\
		0 & 0& 0 & 0 \\
	\end{array}
	\right)P^{T},
	$$
	$$
	\gamma_{3}^{T}C\left(
	\begin{array}{cccc}
		\gamma_{2}&  &  &  \\
		& \gamma_{2} &  &  \\
		& & \gamma_{2}&  \\
		&  &  &\gamma_{2}\\
	\end{array}
	\right)=\gamma_{2}^{T}E_{3}^{T}P^{T}+\gamma_{3}^{T}\left(
	\begin{array}{cccc}
		0 & 0 & 1 & 0 \\
		0 & 0 & 0 & 1 \\
		0& 0 & 0 & 0 \\
		0 & 0& 0 & 0 \\
	\end{array}
	\right)P^{T},
	$$
	which is equivalent to the third equality in  Theorem \ref{thm1}.
	
	When $j=3$,
	$$
	\gamma_{0}^{T}C\left(
	\begin{array}{cccc}
		\gamma_{3}&  &  &  \\
		& \gamma_{3} &  &  \\
		& & \gamma_{3}&  \\
		&  &  &\gamma_{3}\\
	\end{array}
	\right)=\gamma_{3}^{T}E_{0}P^{T}+\gamma_{0}^{T}\left(
	\begin{array}{cccc}
		0 & 0 & 0 & 1 \\
		0 & 0 & 1 & 0 \\
		0 & 0 & 0 & 0 \\
		0 & 0 & 0 & 0 \\
	\end{array}
	\right)P^{T}+\lambda\gamma_{3}^{T},
	$$
	$$
	\gamma_{1}^{T}C\left(
	\begin{array}{cccc}
		\gamma_{3}&  &  &  \\
		& \gamma_{3} &  &  \\
		& & \gamma_{3}&  \\
		&  &  &\gamma_{3}\\
	\end{array}
	\right)=\gamma_{3}^{T}E_{1}^{T}P^{T}+\gamma_{1}^{T}\left(
	\begin{array}{cccc}
		0 & 0 & 0 & 1 \\
		0 & 0 & 1 & 0 \\
		0 & 0 & 0 & 0 \\
		0 & 0 & 0 & 0 \\
	\end{array}
	\right)P^{T}+\lambda\gamma_{2}^{T},
	$$
	$$
	\gamma_{2}^{T}C\left(
	\begin{array}{cccc}
		\gamma_{3}&  &  &  \\
		& \gamma_{3} &  &  \\
		& & \gamma_{3}&  \\
		&  &  &\gamma_{3}\\
	\end{array}
	\right)=\gamma_{3}^{T}E_{2}^{T}P^{T}+\gamma_{2}^{T}\left(
	\begin{array}{cccc}
		0 & 0 & 0 & 1 \\
		0 & 0 & 1 & 0 \\
		0 & 0 & 0 & 0 \\
		0 & 0 & 0 & 0 \\
	\end{array}
	\right)P^{T},
	$$
	$$
	\gamma_{3}^{T}C\left(
	\begin{array}{cccc}
		\gamma_{3}&  &  &  \\
		& \gamma_{3} &  &  \\
		& & \gamma_{3}&  \\
		&  &  &\gamma_{3}\\
	\end{array}
	\right)=\gamma_{3}^{T}E_{3}^{T}P^{T}+\gamma_{3}^{T}\left(
	\begin{array}{cccc}
		0 & 0 & 0 & 1 \\
		0 & 0 & 1 & 0 \\
		0 & 0 & 0 & 0 \\
		0 & 0 & 0 & 0 \\
	\end{array}
	\right)P^{T},
	$$
	which is equivalent to the fourth equality in Theorem \ref{thm1}.
\end{proof}
 
 \section{Rota-Baxter operators with weight 0 on $\mathbb{H}_{ss}$ }
 
 In this section, we shall describe all the Rota-Baxter operators with weight 0 on $\mathbb{H}_{ss}$. Firstly, we should find the all $P$ which satisfy the matrix equations in Theorem \ref{thm1}. It follows easily from Theorem \ref{thm1} that  the system of the matrix equations is equivalent to the system of equations 
 
\begin{equation}\label{ch1}\tag{a1}
-a_{11}^2+a_{21}^2-2 a_{12} a_{21}-2 a_{13} a_{31}-2 a_{14} a_{41}=0,
\end{equation}
\begin{equation}\label{ch2}\tag{a2}
-a_{11} a_{12}-a_{22} a_{12}-a_{11} a_{21}+a_{21} a_{22}-a_{14} a_{31}-a_{13} a_{32}-a_{13} a_{41}-a_{14} a_{42}=0,
\end{equation}
\begin{equation}\label{ch3}\tag{a3}
-a_{11} a_{13}-a_{33} a_{13}+a_{14} a_{21}-a_{12} a_{23}+a_{21} a_{23}-a_{14} a_{43}=0,
\end{equation}
\begin{equation}\label{ch4}\tag{a4}
-a_{11} a_{14}-a_{44} a_{14}+a_{13} a_{21}-a_{12} a_{24}+a_{21} a_{24}-a_{13} a_{34}=0,
	\end{equation}
	\begin{equation}\label{ch5}\tag{a5}
		-2 a_{21} a_{22}-2 a_{23} a_{31}-2 a_{24} a_{41}=0,
\end{equation}
\begin{equation}\label{ch6}\tag{a6}-a_{21}^2-a_{22}^2-a_{24} a_{31}-a_{23} a_{32}-a_{23} a_{41}-a_{24} a_{42}=0,
		\end{equation}
		\begin{equation}\label{ch7}\tag{a7}-a_{22} a_{23}-a_{33} a_{23}+a_{21} a_{24}-a_{24} a_{43}=0,
		\end{equation}
		\begin{equation}\label{ch8}\tag{a8}
			a_{21} a_{23}-a_{34} a_{23}-a_{22} a_{24}-a_{24} a_{44}=0,
			\end{equation}
			\begin{equation}\label{ch9}\tag{a9}
				-2 a_{21} a_{32}-2 a_{31} a_{33}-2 a_{34} a_{41}=0,
			\end{equation}
			\begin{equation}\label{ch10}\tag{a10}
				-a_{21} a_{31}-a_{34} a_{31}-a_{22} a_{32}-a_{32} a_{33}+a_{22} a_{41}-a_{33} a_{41}-a_{21} a_{42}-a_{34} a_{42}=0,
				\end{equation}
				\begin{equation}\label{ch11}\tag{a11}
					-a_{33}^2-a_{23} a_{32}+a_{21} a_{34}+a_{23} a_{41}-a_{21} a_{43}-a_{34} a_{43}=0,
					\end{equation}
					\begin{equation}\label{ch12}\tag{a12}-a_{24} a_{32}+a_{21} a_{33}-a_{33} a_{34}+a_{24} a_{41}-a_{21} a_{44}-a_{34} a_{44}=0,
					\end{equation}
					\begin{equation}\label{ch13}\tag{a13}-2 a_{21} a_{42}-2 a_{31} a_{43}-2 a_{41} a_{44}=0,
					\end{equation}
					\begin{equation}\label{ch14}\tag{a14}a_{22} a_{31}-a_{44} a_{31}-a_{21} a_{32}-a_{21} a_{41}-a_{22} a_{42}-a_{32} a_{43}-a_{41} a_{43}-a_{42} a_{44}=0,
					\end{equation}
					\begin{equation}\label{ch15}\tag{a15}a_{23} a_{31}-a_{21} a_{33}-a_{23} a_{42}-a_{33} a_{43}+a_{21} a_{44}-a_{43} a_{44}=0,
					\end{equation}
					\begin{equation}\label{ch16}\tag{a16}
						-a_{44}^2+a_{24} a_{31}-a_{21} a_{34}-a_{24} a_{42}+a_{21} a_{43}-a_{34} a_{43}=0,
						\end{equation}
						\begin{equation}\label{ch17}\tag{a17}-a_{11} a_{12}-a_{22} a_{12}-a_{11} a_{21}+a_{21} a_{22}+a_{14} a_{31}-a_{13} a_{32}+a_{13} a_{41}-a_{14} a_{42}=0,
						\end{equation}
						\begin{equation}\label{ch18}\tag{a18}-a_{12}^2+a_{22}^2-2 a_{11} a_{22}=0,
						\end{equation}
						\begin{equation}\label{ch19}\tag{a19}-a_{12} a_{13}+a_{43} a_{13}+a_{14} a_{22}-a_{11} a_{23}+a_{22} a_{23}+a_{14} a_{33}=0,
						\end{equation}\begin{equation}\label{ch20}\tag{a20}-a_{12} a_{14}+a_{34} a_{14}+a_{13} a_{22}-a_{11} a_{24}+a_{22} a_{24}+a_{13} a_{44}=0,
						\end{equation}
						\begin{equation}\label{ch21}\tag{a21}-a_{21}^2-a_{22}^2+a_{24} a_{31}-a_{23} a_{32}+a_{23} a_{41}-a_{24} a_{42}=0,
						\end{equation}
						\begin{equation}\label{ch22}\tag{a22}-2 a_{21} a_{22}=0,
						\end{equation}
						\begin{equation}\label{ch23}\tag{a23}-a_{21} a_{23}+a_{43} a_{23}+a_{22} a_{24}+a_{24} a_{33}=0,
						\end{equation}
						\begin{equation}\label{ch24}\tag{a24}a_{22} a_{23}+a_{44} a_{23}-a_{21} a_{24}+a_{24} a_{34}=0,
						\end{equation}
						\begin{equation}\label{ch25}\tag{a25}-a_{21} a_{31}+a_{34} a_{31}-a_{22} a_{32}-a_{32} a_{33}-a_{22} a_{41}+a_{33} a_{41}+a_{21} a_{42}-a_{34} a_{42}=0,
						\end{equation}
						\begin{equation}\label{ch26}\tag{a26}-2 a_{22} a_{31}=0,
						\end{equation}
						\begin{equation}\label{ch27}\tag{a27}-a_{23} a_{31}+a_{22} a_{34}+a_{33} a_{34}+a_{23} a_{42}-a_{22} a_{43}+a_{33} a_{43}=0,
						\end{equation}
						\begin{equation}\label{ch28}\tag{a28}
							a_{34}^2-a_{24} a_{31}+a_{22} a_{33}+a_{24} a_{42}-a_{22} a_{44}+a_{33} a_{44}=0,
							\end{equation}
							\begin{equation}\label{ch29}\tag{a29}-a_{22} a_{31}+a_{44} a_{31}+a_{21} a_{32}-a_{21} a_{41}-a_{22} a_{42}-a_{32} a_{43}+a_{41} a_{43}-a_{42} a_{44}=0,
							\end{equation}
							\begin{equation}\label{ch30}\tag{a30}-2 a_{22} a_{41}=0,
							\end{equation}
							\begin{equation}\label{ch31}\tag{a31}a_{43}^2+a_{23} a_{32}-a_{22} a_{33}-a_{23} a_{41}+a_{22} a_{44}+a_{33} a_{44}=0,
							\end{equation}
							\begin{equation}\label{ch32}\tag{a32}a_{24} a_{32}-a_{22} a_{34}-a_{24} a_{41}+a_{22} a_{43}+a_{34} a_{44}+a_{43} a_{44}=0,
							\end{equation}
							\begin{equation}\label{ch33}\tag{a33}-a_{11} a_{13}-a_{33} a_{13}-a_{14} a_{21}-a_{12} a_{23}+a_{21} a_{23}-a_{14} a_{43}=0,
							\end{equation}
							\begin{equation}\label{ch34}\tag{a34}-a_{12} a_{13}-a_{43} a_{13}-a_{14} a_{22}-a_{11} a_{23}+a_{22} a_{23}-a_{14} a_{33}=0,
							\end{equation}
							\begin{equation}\label{ch35}\tag{a35}a_{23}^2-a_{13}^2=0,
							\end{equation}
							\begin{equation}\label{ch36}\tag{a36}-a_{13} a_{14}-a_{24} a_{14}+a_{13} a_{23}+a_{23} a_{24}=0,
							\end{equation}\begin{equation}\label{ch37}\tag{a37}-a_{22} a_{23}-a_{33} a_{23}-a_{21} a_{24}-a_{24} a_{43}=0,
							\end{equation}
							\begin{equation}\label{ch38}\tag{a38}-a_{21} a_{23}-a_{43} a_{23}-a_{22} a_{24}-a_{24} a_{33}=0,
							\end{equation}
							\begin{equation}\label{ch39}\tag{a39}a_{23}^2-a_{24}^2=0,
							\end{equation}
							\begin{equation}\label{ch40}\tag{a40}-a_{33}^2-a_{23} a_{32}-a_{21} a_{34}-a_{23} a_{41}+a_{21} a_{43}-a_{34} a_{43}=0,
							\end{equation}
							\begin{equation}\label{ch41}\tag{a41}-a_{23} a_{31}-a_{22} a_{34}-a_{33} a_{34}-a_{23} a_{42}+a_{22} a_{43}-a_{33} a_{43}=0,
							\end{equation}
							\begin{equation}\label{ch42}\tag{a42}a_{23} a_{33}-a_{24} a_{34}+a_{24} a_{43}-a_{23} a_{44}=0,
							\end{equation}
							\begin{equation}\label{ch43}\tag{a43}-a_{23} a_{31}+a_{21} a_{33}-a_{23} a_{42}-a_{33} a_{43}-a_{21} a_{44}-a_{43} a_{44}=0,
							\end{equation}
							\begin{equation}\label{ch44}\tag{a44}-a_{43}^2-a_{23} a_{32}+a_{22} a_{33}-a_{23} a_{41}-a_{22} a_{44}-a_{33} a_{44}=0,
							\end{equation}
							\begin{equation}\label{ch45}\tag{a45}a_{24} a_{33}-a_{23} a_{34}+a_{23} a_{43}-a_{24} a_{44}=0,
							\end{equation}
							\begin{equation}\label{ch46}\tag{a46}-a_{11} a_{14}-a_{44} a_{14}-a_{13} a_{21}-a_{12} a_{24}+a_{21} a_{24}-a_{13} a_{34}=0,
							\end{equation}
							\begin{equation}\label{ch47}\tag{a47}-a_{12} a_{14}-a_{34} a_{14}-a_{13} a_{22}-a_{11} a_{24}+a_{22} a_{24}-a_{13} a_{44}=0,
							\end{equation}
							\begin{equation}\label{ch48}\tag{a48}-a_{13} a_{14}+a_{24} a_{14}-a_{13} a_{23}+a_{23} a_{24}=0,
							\end{equation}
							\begin{equation}\label{ch49}\tag{a49}a_{24}^2-a_{14}^2=0,
							\end{equation}
							\begin{equation}\label{ch50}\tag{a50}-a_{21} a_{23}-a_{34} a_{23}-a_{22} a_{24}-a_{24} a_{44}=0,
							\end{equation}
							\begin{equation}\label{ch51}\tag{a51}-a_{22} a_{23}-a_{44} a_{23}-a_{21} a_{24}-a_{24} a_{34}=0,
							\end{equation}
							\begin{equation}\label{ch52}\tag{a52}a_{24}^2-a_{23}^2=0,
							\end{equation}
							\begin{equation}\label{ch53}\tag{a53}-a_{24} a_{32}-a_{21} a_{33}-a_{33} a_{34}-a_{24} a_{41}+a_{21} a_{44}-a_{34} a_{44}=0,
							\end{equation}
							\begin{equation}\label{ch54}\tag{a54}-a_{34}^2-a_{24} a_{31}-a_{22} a_{33}-a_{24} a_{42}+a_{22} a_{44}-a_{33} a_{44}=0,
							\end{equation}
							\begin{equation}\label{ch55}\tag{a55}-a_{23} a_{33}+a_{24} a_{34}-a_{24} a_{43}+a_{23} a_{44}=0,
							\end{equation}
							\begin{equation}\label{ch56}\tag{a56}-a_{44}^2-a_{24} a_{31}+a_{21} a_{34}-a_{24} a_{42}-a_{21} a_{43}-a_{34} a_{43}=0,
							\end{equation}
							\begin{equation}\label{ch57}\tag{a57}-a_{24} a_{32}+a_{22} a_{34}-a_{24} a_{41}-a_{22} a_{43}-a_{34} a_{44}-a_{43} a_{44}=0,
							\end{equation}
							\begin{equation}\label{ch58}\tag{a58}-a_{24} a_{33}+a_{23} a_{34}-a_{23} a_{43}+a_{24} a_{44}=0.
							\end{equation}

It is difficult to finding all solutions for the system of equations (\ref{ch1})-(\ref{ch58}) by using  manual methods. With the help of scientific computation software-\textbf{Mathematica}, by inputting the order 
\textbf{Reduce[$(\text{a1}),(\text{a2}),\cdots,(\text{a58}), \{a_{11},a_{12},\cdots,\cdots,a_{43},a_{44}\}$]}, we can get the main result in this section. 

\begin{theorem}	\label{k2}$\mathcal{P}$ is a Rota-Baxter operator of weight 0 on $\mathbb{H}_{ss}$ if and only if  $P$ is  one of the following matrices 
	$$
	\left(
	\begin{array}{cccc}
		0 & 0 & 0 & 0 \\
		0 & 0 & 0 & 0 \\
		a & b & 0 & 0 \\
		c & d & 0 & 0 \\
	\end{array}
	\right), \left(
	\begin{array}{cccc}
		0 & 0 & 0 & 0 \\
		0 & 0 & 0 & 0 \\
		a & b & c & -c \\
		a & b & c & -c \\
	\end{array}
	\right),\left(
	\begin{array}{cccc}
		0 & 0 & 0 & 0 \\
		0 & 0 & 0 & 0 \\
		a & b & c & c \\
		-a & -b & -c & -c \\
	\end{array}
	\right),
	$$
	$$
	\left(
	\begin{array}{cccc}
		0 & 0 & a & -a \\
		0 & 0 & -a & a \\
		0 & 0 & b & -b \\
		0 & 0 & b & -b \\
	\end{array}
	\right),\left(
	\begin{array}{cccc}
		0 & 0 & a & -a \\
		0 & 0 & a & -a \\
		0 & 0 & b & -b \\
		0 & 0 & b & -b \\
	\end{array}
	\right),\left(
	\begin{array}{cccc}
		0 & 0 & a & a \\
		0 & 0 & -a & -a \\
		0 & 0 & b & b \\
		0 & 0 & -b & -b \\
	\end{array}
	\right),
	$$
	$$
	\left(
	\begin{array}{cccc}
		0 & 0 & a & a \\
		0 & 0 & a & a \\
		0 & 0 & b & b \\
		0 & 0 & -b & -b \\
	\end{array}
	\right),
\left(
	\begin{array}{cccc}
		0 & a & b & -b \\
		0 & -a & -b & b \\
		0 & c & \frac{b c}{a} & -\frac{b c}{a} \\
		0 & c-\frac{a^2}{b} & \frac{b c}{a}-a & a-\frac{b c}{a} \\
	\end{array}
	\right),
	$$
	$$
	\left(
	\begin{array}{cccc}
		0 & a & b & -b \\
		0 & a & b & -b \\
		0 & c & \frac{b c}{a} & -\frac{b c}{a} \\
		0 & \frac{a^2}{b}+c & \frac{b c}{a}+a & -\frac{a^2+b c}{a} \\
	\end{array}
	\right),\left(
	\begin{array}{cccc}
		0 & a & b & b \\
		0 & -a & -b & -b \\
		0 & c & \frac{b c}{a} & \frac{b c}{a} \\
		0 & \frac{a^2}{b}-c & a-\frac{b c}{a} & a-\frac{b c}{a} \\
	\end{array}
	\right),
	$$
	$$
	\left(
	\begin{array}{cccc}
		0 & a & b & b \\
		0 & a & b & b \\
		0 & c & \frac{b c}{a} & \frac{b c}{a} \\
		0 & -\frac{a^2+b c}{b} & -\frac{a^2+b c}{a} & -\frac{a^2+b c}{a} \\
	\end{array}
	\right),
	$$
where $a,b,c$ are parameters.
\end{theorem}

 \section{Rota-Baxter operators with non-zero weight $\lambda$ on $\mathbb{H}_{ss}$ }
 
 In this section, we shall describe all Rota-Baxter operators with non-zero weight on $\mathbb{H}_{ss}$. It follows easily from Theorem \ref{thm1} that  the system of the matrix equations is equivalent to the system of equations 
  \begin{equation}
  	\label{q1}\tag{b1}
 -a_{11}^2+a_{21}^2-2 a_{12} a_{21}-2 a_{13} a_{31}-2 a_{14} a_{41}-\lambda  a _{11}=0, 
\end{equation}
\begin{equation}\label{q2}\tag{b2}-a_{11} a_{12}-a_{22} a_{12}-a_{11} a_{21}+a_{21} a_{22}-a_{14} a_{31}-a_{13} a_{32}-a_{13} a_{41}-a_{14} a_{42}-\lambda  a _{12}=0, 
\end{equation}
\begin{equation}\label{q3}\tag{b3}-a_{11} a_{13}-a_{33} a_{13}+a_{14} a_{21}-a_{12} a_{23}+a_{21} a_{23}-a_{14} a_{43}-\lambda  a _{13}=0, \end{equation}
\begin{equation}\label{q4}\tag{b4}-a_{11} a_{14}-a_{44} a_{14}+a_{13} a_{21}-a_{12} a_{24}+a_{21} a_{24}-a_{13} a_{34}-\lambda  a _{14}=0, \end{equation}
\begin{equation}\label{q5}\tag{b5}-2 a_{21} a_{22}-2 a_{23} a_{31}-2 a_{24} a_{41}-\lambda  a _{21}=0, \end{equation}
\begin{equation}\label{q6}\tag{b6}-a_{21}^2-a_{22}^2-a_{24} a_{31}-a_{23} a_{32}-a_{23} a_{41}-a_{24} a_{42}-\lambda  a _{22}=0,
 \end{equation}
\begin{equation}\label{q7}\tag{b7}-a_{22} a_{23}-a_{33} a_{23}+a_{21} a_{24}-a_{24} a_{43}-\lambda  a _{23}=0, \end{equation}
\begin{equation}\label{q8}\tag{b8}a_{21} a_{23}-a_{34} a_{23}-a_{22} a_{24}-a_{24} a_{44}-\lambda  a _{24}=0, \end{equation}
\begin{equation}\label{q9}\tag{b9}-2 a_{21} a_{32}-2 a_{31} a_{33}-2 a_{34} a_{41}-\lambda  a _{31}=0, \end{equation}
\begin{equation}\label{q10}\tag{b10}-a_{21} a_{31}-a_{34} a_{31}-a_{22} a_{32}-a_{32} a_{33}+a_{22} a_{41}-a_{33} a_{41}-a_{21} a_{42}-a_{34} a_{42}-\lambda  a _{32}=0, \end{equation}
\begin{equation}\label{q11}\tag{b11}-a_{33}^2-a_{23} a_{32}+a_{21} a_{34}+a_{23} a_{41}-a_{21} a_{43}-a_{34} a_{43}-\lambda  a _{33}=0, \end{equation}
\begin{equation}\label{q12}\tag{b12}-a_{24} a_{32}+a_{21} a_{33}-a_{33} a_{34}+a_{24} a_{41}-a_{21} a_{44}-a_{34} a_{44}-\lambda  a _{34}=0, \end{equation}
\begin{equation}\label{q13}\tag{b13}-2 a_{21} a_{42}-2 a_{31} a_{43}-2 a_{41} a_{44}-\lambda  a _{41}=0, \end{equation}
\begin{equation}\label{q14}\tag{b14}a_{22} a_{31}-a_{44} a_{31}-a_{21} a_{32}-a_{21} a_{41}-a_{22} a_{42}-a_{32} a_{43}-a_{41} a_{43}-a_{42} a_{44}-\lambda  a _{42}=0, \end{equation}
\begin{equation}\label{q15}\tag{b15}a_{23} a_{31}-a_{21} a_{33}-a_{23} a_{42}-a_{33} a_{43}+a_{21} a_{44}-a_{43} a_{44}-\lambda  a _{43}=0, \end{equation}
\begin{equation}\label{q16}\tag{b16}-a_{44}^2+a_{24} a_{31}-a_{21} a_{34}-a_{24} a_{42}+a_{21} a_{43}-a_{34} a_{43}-\lambda  a _{44}=0, \end{equation}
\begin{equation}\label{q17}\tag{b17}-a_{11} a_{12}-a_{22} a_{12}-a_{11} a_{21}+a_{21} a_{22}+a_{14} a_{31}-a_{13} a_{32}+a_{13} a_{41}-a_{14} a_{42}-\lambda  a _{12}=0, \end{equation}
\begin{equation}\label{q18}\tag{b18}-a_{12}^2+a_{22}^2-2 a_{11} a_{22}-\lambda  a _{11}=0, \end{equation}
\begin{equation}\label{q19}\tag{b19}-a_{12} a_{13}+a_{43} a_{13}+a_{14} a_{22}-a_{11} a_{23}+a_{22} a_{23}+a_{14} a_{33}+\lambda  a _{14}=0, \end{equation}
\begin{equation}\label{q20}\tag{b20}-a_{12} a_{14}+a_{34} a_{14}+a_{13} a_{22}-a_{11} a_{24}+a_{22} a_{24}+a_{13} a_{44}+\lambda  a _{13}=0, \end{equation}
\begin{equation}\label{q21}\tag{b21}-a_{21}^2-a_{22}^2+a_{24} a_{31}-a_{23} a_{32}+a_{23} a_{41}-a_{24} a_{42}-\lambda  a _{22}=0, \end{equation}
\begin{equation}\label{q22}\tag{b22}-2 a_{21} a_{22}-\lambda  a _{21}=0, \end{equation}
\begin{equation}\label{q23}\tag{b23}-a_{21} a_{23}+a_{43} a_{23}+a_{22} a_{24}+a_{24} a_{33}+\lambda  a _{24}=0, \end{equation}
\begin{equation}\label{q24}\tag{b24}a_{22} a_{23}+a_{44} a_{23}-a_{21} a_{24}+a_{24} a_{34}+\lambda  a _{23}=0, \end{equation}
\begin{equation}\label{q25}\tag{b25}-a_{21} a_{31}+a_{34} a_{31}-a_{22} a_{32}-a_{32} a_{33}-a_{22} a_{41}+a_{33} a_{41}+a_{21} a_{42}-a_{34} a_{42}-\lambda  a _{32}=0, \end{equation}
\begin{equation}\label{q26}\tag{b26}-2 a_{22} a_{31}-\lambda  a _{31}=0, \end{equation}
\begin{equation}\label{q27}\tag{b27}-a_{23} a_{31}+a_{22} a_{34}+a_{33} a_{34}+a_{23} a_{42}-a_{22} a_{43}+a_{33} a_{43}+\lambda  a _{34}=0, \end{equation}
\begin{equation}\label{q28}\tag{b28}a_{34}^2-a_{24} a_{31}+a_{22} a_{33}+a_{24} a_{42}-a_{22} a_{44}+a_{33} a_{44}+\lambda  a _{33}=0, \end{equation}
\begin{equation}\label{q29}\tag{b29}-a_{22} a_{31}+a_{44} a_{31}+a_{21} a_{32}-a_{21} a_{41}-a_{22} a_{42}-a_{32} a_{43}+a_{41} a_{43}-a_{42} a_{44}-\lambda  a _{42}=0,
 \end{equation}
\begin{equation}\label{q30}\tag{b30}-2 a_{22} a_{41}-\lambda  a _{41}=0, \end{equation}
\begin{equation}\label{q31}\tag{b31}a_{43}^2+a_{23} a_{32}-a_{22} a_{33}-a_{23} a_{41}+a_{22} a_{44}+a_{33} a_{44}+\lambda  a _{44}=0, \end{equation}
\begin{equation}\label{q32}\tag{b32}a_{24} a_{32}-a_{22} a_{34}-a_{24} a_{41}+a_{22} a_{43}+a_{34} a_{44}+a_{43} a_{44}+\lambda  a _{43}=0, \end{equation}
\begin{equation}\label{q33}\tag{b33}-a_{11} a_{13}-a_{33} a_{13}-a_{14} a_{21}-a_{12} a_{23}+a_{21} a_{23}-a_{14} a_{43}-\lambda  a _{13}=0, \end{equation}
\begin{equation}\label{q34}\tag{b34}-a_{12} a_{13}-a_{43} a_{13}-a_{14} a_{22}-a_{11} a_{23}+a_{22} a_{23}-a_{14} a_{33}-\lambda  a _{14}=0, \end{equation}
\begin{equation}\label{q35}\tag{b35}a_{35}^2-a_{13}^2=0, \end{equation}\begin{equation}\label{q36}\tag{b36}-a_{13} a_{14}-a_{24} a_{14}+a_{13} a_{23}+a_{23} a_{24}=0,
 \end{equation}
\begin{equation}\label{q37}\tag{b37}-a_{22} a_{23}-a_{33} a_{23}-a_{21} a_{24}-a_{24} a_{43}-\lambda  a _{23}=0,
 \end{equation}
\begin{equation}\label{q38}\tag{b38}-a_{21} a_{23}-a_{43} a_{23}-a_{22} a_{24}-a_{24} a_{33}-\lambda  a _{24}=0, \end{equation}
\begin{equation}\label{q39}\tag{b39}a_{23}^2-a_{24}^2=0,
 \end{equation}
\begin{equation}\label{q40}\tag{b40}-a_{33}^2-a_{23} a_{32}-a_{21} a_{34}-a_{23} a_{41}+a_{21} a_{43}-a_{34} a_{43}-\lambda  a _{33}=0, \end{equation}
\begin{equation}\label{q41}\tag{b41}-a_{23} a_{31}-a_{22} a_{34}-a_{33} a_{34}-a_{23} a_{42}+a_{22} a_{43}-a_{33} a_{43}-\lambda  a _{34}=0, \end{equation}
\begin{equation}\label{q42}\tag{b42}a_{23} a_{33}-a_{24} a_{34}+a_{24} a_{43}-a_{23} a_{44}=0, \end{equation}
\begin{equation}\label{q43}\tag{b43}-a_{23} a_{31}+a_{21} a_{33}-a_{23} a_{42}-a_{33} a_{43}-a_{21} a_{44}-a_{43} a_{44}-\lambda  a _{43}=0, \end{equation}
\begin{equation}\label{q44}\tag{b44}-a_{43}^2-a_{23} a_{32}+a_{22} a_{33}-a_{23} a_{41}-a_{22} a_{44}-a_{33} a_{44}-\lambda  a _{44}=0, \end{equation}
\begin{equation}\label{q45}\tag{b45}a_{24} a_{33}-a_{23} a_{34}+a_{23} a_{43}-a_{24} a_{44}=0, \end{equation}
\begin{equation}\label{q46}\tag{b46}-a_{11} a_{14}-a_{44} a_{14}-a_{13} a_{21}-a_{12} a_{24}+a_{21} a_{24}-a_{13} a_{34}-\lambda  a _{14}=0, \end{equation}
\begin{equation}\label{q47}\tag{b47}-a_{12} a_{14}-a_{34} a_{14}-a_{13} a_{22}-a_{11} a_{24}+a_{22} a_{24}-a_{13} a_{44}-\lambda  a _{13}=0, \end{equation}
\begin{equation}\label{q48}\tag{b48}-a_{13} a_{14}+a_{24} a_{14}-a_{13} a_{23}+a_{23} a_{24}=0, \end{equation}
\begin{equation}\label{q49}\tag{b49}a_{24}^2-a_{14}^2=0, \end{equation}
\begin{equation}\label{q50}\tag{b50}-a_{21} a_{23}-a_{34} a_{23}-a_{22} a_{24}-a_{24} a_{44}-\lambda  a _{24}=0, \end{equation}
\begin{equation}\label{q51}\tag{b51}-a_{22} a_{23}-a_{44} a_{23}-a_{21} a_{24}-a_{24} a_{34}-\lambda  a _{23}=0, \end{equation}
\begin{equation}\label{q52}\tag{b52}a_{24}^2-a_{23}^2=0, \end{equation}
\begin{equation}\label{q53}\tag{b53}-a_{24} a_{32}-a_{21} a_{33}-a_{33} a_{34}-a_{24} a_{41}+a_{21} a_{44}-a_{34} a_{44}-\lambda  a _{34}=0, \end{equation}
\begin{equation}\label{q54}\tag{b54}-a_{34}^2-a_{24} a_{31}-a_{22} a_{33}-a_{24} a_{42}+a_{22} a_{44}-a_{33} a_{44}-\lambda  a _{33}=0, \end{equation}
\begin{equation}\label{q55}\tag{b55}-a_{23} a_{33}+a_{24} a_{34}-a_{24} a_{43}+a_{23} a_{44}=0, \end{equation}
\begin{equation}\label{q56}\tag{b56}-a_{44}^2-a_{24} a_{31}+a_{21} a_{34}-a_{24} a_{42}-a_{21} a_{43}-a_{34} a_{43}-\lambda  a _{44}=0, \end{equation}
\begin{equation}\label{q57}\tag{b57}-a_{24} a_{32}+a_{22} a_{34}-a_{24} a_{41}-a_{22} a_{43}-a_{34} a_{44}-a_{43} a_{44}-\lambda  a _{43}=0, \end{equation}
\begin{equation}\label{q58}\tag{b58}-a_{24} a_{33}+a_{23} a_{34}-a_{23} a_{43}+a_{24} a_{44}=0.\end{equation}

By inputting the order 
\textbf{Reduce[$(\text{b1}),(\text{b2}),\cdots,(\text{b58}),\{a_{11},a_{12},\cdots,\cdots,a_{43},a_{44}\}$]}, we can get the main result in this section.

\begin{theorem}\label{k1} 	$\mathcal{P}$ is a Rota-Baxter operator of non-zero weight $\lambda$ on $\mathbb{H}_{ss}$ if and only if  $P$ is  one of the following matrices
	$$
	\left(
	\begin{array}{cccc}
		0 & -\lambda  & 0 & 0 \\
		0 & -\lambda  & 0 & 0 \\
		0 & a & 0 & 0 \\
		0 & b & 0 & 0 \\
	\end{array}
	\right),\left(
	\begin{array}{cccc}
		0 & \lambda  & 0 & 0 \\
		0 & -\lambda  & 0 & 0 \\
		0 & a & 0 & 0 \\
		0 & b & 0 & 0 \\
	\end{array}
	\right),\left(
	\begin{array}{cccc}
		-\lambda  & 0 & 0 & 0 \\
		0 & -\lambda  & 0 & 0 \\
		0 & a & 0 & 0 \\
		0 & b & 0 & 0 \\
	\end{array}
	\right),
	$$
	$$
	\left(
	\begin{array}{cccc}
		0 & 0 & 0 & 0 \\
		0 & 0 & 0 & 0 \\
		0 & a & -\lambda  & 0 \\
		0 & b & 0 & -\lambda  \\
	\end{array}
	\right),\left(
	\begin{array}{cccc}
		-\lambda  & -\lambda  & 0 & 0 \\
		0 & 0 & 0 & 0 \\
		0 & a & -\lambda  & 0 \\
		0 & b & 0 & -\lambda  \\
	\end{array}
	\right),\left(
	\begin{array}{cccc}
		-\lambda  & \lambda  & 0 & 0 \\
		0 & 0 & 0 & 0 \\
		0 & a & -\lambda  & 0 \\
		0 & b & 0 & -\lambda  \\
	\end{array}
	\right),
	$$
$$
\left(
\begin{array}{cccc}
	0 & -\lambda  & 0 & 0 \\
	0 & -\lambda  & 0 & 0 \\
	0 & 0 & -\lambda  & 0 \\
	0 & a & -\lambda  & 0 \\
\end{array}
\right),\left(
\begin{array}{cccc}
	0 & \lambda  & 0 & 0 \\
	0 & -\lambda  & 0 & 0 \\
	0 & 0 & -\lambda  & 0 \\
	0 & a & -\lambda  & 0 \\
\end{array}
\right),\left(
\begin{array}{cccc}
	0 & -\lambda  & 0 & 0 \\
	0 & -\lambda  & 0 & 0 \\
	0 & 0 & -\lambda  & 0 \\
	0 & a & \lambda  & 0 \\
\end{array}
\right),\left(
\begin{array}{cccc}
	0 & \lambda  & 0 & 0 \\
	0 & -\lambda  & 0 & 0 \\
	0 & 0 & -\lambda  & 0 \\
	0 & a & \lambda  & 0 \\
\end{array}
\right),
$$
	$$
\left(
\begin{array}{cccc}
	-\lambda  & 0 & 0 & 0 \\
	0 & -\lambda  & 0 & 0 \\
	0 & 0 & -\lambda  & 0 \\
	0 & a & -\lambda  & 0 \\
\end{array}
\right),\left(
\begin{array}{cccc}
	-\lambda  & 0 & 0 & 0 \\
	0 & -\lambda  & 0 & 0 \\
	0 & 0 & -\lambda  & 0 \\
	0 & a & \lambda  & 0 \\
\end{array}
\right),
\left(
\begin{array}{cccc}
	0 & 0 & 0 & 0 \\
	0 & 0 & 0 & 0 \\
	0 & 0 & 0 & 0 \\
	0 & a & -\lambda  & -\lambda  \\
\end{array}
\right),\left(
\begin{array}{cccc}
	0 & 0 & 0 & 0 \\
	0 & 0 & 0 & 0 \\
	0 & 0 & 0 & 0 \\
	0 & a & \lambda  & -\lambda  \\
\end{array}
\right),	
$$
$$
\left(
\begin{array}{cccc}
	-\lambda  & -\lambda  & 0 & 0 \\
	0 & 0 & 0 & 0 \\
	0 & 0 & 0 & 0 \\
	0 & a & -\lambda  & -\lambda  \\
\end{array}
\right),\left(
\begin{array}{cccc}
	-\lambda  & \lambda  & 0 & 0 \\
	0 & 0 & 0 & 0 \\
	0 & 0 & 0 & 0 \\
	0 & a & -\lambda  & -\lambda  \\
\end{array}
\right),\left(
\begin{array}{cccc}
	-\lambda  & -\lambda  & 0 & 0 \\
	0 & 0 & 0 & 0 \\
	0 & 0 & 0 & 0 \\
	0 & a & \lambda  & -\lambda  \\
\end{array}
\right),\left(
\begin{array}{cccc}
	-\lambda  & \lambda  & 0 & 0 \\
	0 & 0 & 0 & 0 \\
	0 & 0 & 0 & 0 \\
	0 & a & \lambda  & -\lambda  \\
\end{array}
\right),
$$	
	$$
	\left(
	\begin{array}{cccc}
		0 & -b & a & -a \\
		0 & b & -a & a \\
		0 & \frac{(b-c) (b+\lambda )}{a} & -b+c-\lambda  & b-c \\
		0 & -\frac{c (b+\lambda )}{a} & c & -c-\lambda  \\
	\end{array}
	\right),\left(
	\begin{array}{cccc}
		0 & -b & a & -a \\
		0 & b & -a & a \\
		0 & \frac{b (b-c+\lambda )}{a} & -b+c-\lambda  & b-c+\lambda  \\
		0 & -\frac{b c}{a} & c & -c \\
	\end{array}
	\right),
	$$
	$$
	\left(
	\begin{array}{cccc}
		-\lambda  & -b-\lambda  & a & -a \\
		0 & b & -a & a \\
		0 & \frac{(b-c) (b+\lambda )}{a} & -b+c-\lambda  & b-c \\
		0 & -\frac{c (b+\lambda )}{a} & c & -c-\lambda  \\
	\end{array}
	\right),\left(
	\begin{array}{cccc}
		-\lambda  & -b-\lambda  & a & -a \\
		0 & b & -a & a \\
		0 & \frac{b (b-c+\lambda )}{a} & -b+c-\lambda  & b-c+\lambda  \\
		0 & -\frac{b c}{a} & c & -c \\
	\end{array}
	\right),
	$$
	$$
	\left(
	\begin{array}{cccc}
		0 & b & a & -a \\
		0 & b & a & -a \\
		0 & -\frac{(b-c) (b+\lambda )}{a} & -b+c-\lambda  & b-c \\
		0 & \frac{c (b+\lambda )}{a} & c & -c-\lambda  \\
	\end{array}
	\right),\left(
	\begin{array}{cccc}
		0 & b & a & -a \\
		0 & b & a & -a \\
		0 & -\frac{b (b-c+\lambda )}{a} & -b+c-\lambda  & b-c+\lambda  \\
		0 & \frac{b c}{a} & c & -c \\
	\end{array}
	\right),
	$$
	$$
	\left(
	\begin{array}{cccc}
		-\lambda  & b+\lambda  & a & -a \\
		0 & b & a & -a \\
		0 & -\frac{(b-c) (b+\lambda )}{a} & -b+c-\lambda  & b-c \\
		0 & \frac{c (b+\lambda )}{a} & c & -c-\lambda  \\
	\end{array}
	\right),\left(
	\begin{array}{cccc}
		-\lambda  & b+\lambda  & a & -a \\
		0 & b & a & -a \\
		0 & -\frac{b (b-c+\lambda )}{a} & -b+c-\lambda  & b-c+\lambda  \\
		0 & \frac{b c}{a} & c & -c \\
	\end{array}
	\right),
	$$
	$$
	\left(
	\begin{array}{cccc}
		0 & -b & a & a \\
		0 & b & -a & -a \\
		0 & \frac{(b+c) (b+\lambda )}{a} & -b-c-\lambda  & -b-c \\
		0 & -\frac{c (b+\lambda )}{a} & c & c-\lambda  \\
	\end{array}
	\right), \left(
	\begin{array}{cccc}
		0 & -b & a & a \\
		0 & b & -a & -a \\
		0 & \frac{b (b+c+\lambda )}{a} & -b-c-\lambda  & -b-c-\lambda  \\
		0 & -\frac{b c}{a} & c & c \\
	\end{array}
	\right),
	$$
	$$
	\left(
	\begin{array}{cccc}
		-\lambda  & -b-\lambda  & a & a \\
		0 & b & -a & -a \\
		0 & \frac{(b+c) (b+\lambda )}{a} & -b-c-\lambda  & -b-c \\
		0 & -\frac{c (b+\lambda )}{a} & c & c-\lambda  \\
	\end{array}
	\right),\left(
	\begin{array}{cccc}
		-\lambda  & -b-\lambda  & a & a \\
		0 & b & -a & -a \\
		0 & \frac{b (b+c+\lambda )}{a} & -b-c-\lambda  & -b-c-\lambda  \\
		0 & -\frac{b c}{a} & c & c \\
	\end{array}
	\right),
	$$
	$$
	\left(
	\begin{array}{cccc}
		0 & b & a & a \\
		0 & b & a & a \\
		0 & -\frac{(b+c) (b+\lambda )}{a} & -b-c-\lambda  & -b-c \\
		0 & \frac{c (b+\lambda )}{a} & c & c-\lambda  \\
	\end{array}
	\right),\left(
	\begin{array}{cccc}
		0 & b & a & a \\
		0 & b & a & a \\
		0 & -\frac{b (b+c+\lambda )}{a} & -b-c-\lambda  & -b-c-\lambda  \\
		0 & \frac{b c}{a} & c & c \\
	\end{array}
	\right),
	$$
	$$
	\left(
	\begin{array}{cccc}
		-\lambda  & b+\lambda  & a & a \\
		0 & b & a & a \\
		0 & -\frac{(b+c) (b+\lambda )}{a} & -b-c-\lambda  & -b-c \\
		0 & \frac{c (b+\lambda )}{a} & c & c-\lambda  \\
	\end{array}
	\right),\left(
	\begin{array}{cccc}
		-\lambda  & b+\lambda  & a & a \\
		0 & b & a & a \\
		0 & -\frac{b (b+c+\lambda )}{a} & -b-c-\lambda  & -b-c-\lambda  \\
		0 & \frac{b c}{a} & c & c \\
	\end{array}
	\right),
	$$

	$$
	\left(
	\begin{array}{cccc}
		0 & 0 & 0 & 0 \\
		0 & 0 & 0 & 0 \\
		0 & 0 & 0 & 0 \\
		0 & 0 & 0 & 0 \\
	\end{array}
	\right),\left(
	\begin{array}{cccc}
		-\lambda  & 0 & 0 & 0 \\
		0 & -\lambda  & 0 & 0 \\
		0 & 0 & -\lambda  & 0 \\
		0 & 0 & 0 & -\lambda  \\
	\end{array}
	\right),
	\left(
	\begin{array}{cccc}
		-\lambda  & -\lambda  & 0 & 0 \\
		0 & 0 & 0 & 0 \\
		0 & 0 & 0 & 0 \\
		0 & 0 & 0 & 0 \\
	\end{array}
	\right),\left(
	\begin{array}{cccc}
		-\lambda  & \lambda  & 0 & 0 \\
		0 & 0 & 0 & 0 \\
		0 & 0 & 0 & 0 \\
		0 & 0 & 0 & 0 \\
	\end{array}
	\right),
	$$

	$$
	\left(
	\begin{array}{cccc}
		0 & -\lambda  & 0 & 0 \\
		0 & -\lambda  & 0 & 0 \\
		0 & 0 & -\lambda  & 0 \\
		0 & 0 & 0 & -\lambda  \\
	\end{array}
	\right),\left(
	\begin{array}{cccc}
		0 & \lambda  & 0 & 0 \\
		0 & -\lambda  & 0 & 0 \\
		0 & 0 & -\lambda  & 0 \\
		0 & 0 & 0 & -\lambda  \\
	\end{array}
	\right),
	\left(
	\begin{array}{cccc}
		0 & -\lambda  & 0 & 0 \\
		0 & -\lambda  & 0 & 0 \\
		0 & a & b & -b-\lambda  \\
		0 & \frac{a b}{b+\lambda } & b & -b-\lambda  \\
	\end{array}
	\right),
	$$$$\left(
	\begin{array}{cccc}
		0 & \lambda  & 0 & 0 \\
		0 & -\lambda  & 0 & 0 \\
		0 & a & b & -b-\lambda  \\
		0 & \frac{a b}{b+\lambda } & b & -b-\lambda  \\
	\end{array}
	\right),\left(
	\begin{array}{cccc}
		0 & -\lambda  & 0 & 0 \\
		0 & -\lambda  & 0 & 0 \\
		0 & a & -b & \lambda -b \\
		0 & \frac{a b}{\lambda -b} & b & b-\lambda  \\
	\end{array}
	\right),\left(
	\begin{array}{cccc}
		0 & \lambda  & 0 & 0 \\
		0 & -\lambda  & 0 & 0 \\
		0 & a & -b & \lambda -b \\
		0 & \frac{a b}{\lambda -b} & b & b-\lambda  \\
	\end{array}
	\right),
	$$
	$$
	\left(
	\begin{array}{cccc}
		-\lambda  & 0 & 0 & 0 \\
		0 & -\lambda  & 0 & 0 \\
		0 & a & b & -b-\lambda  \\
		0 & \frac{a b}{b+\lambda } & b & -b-\lambda  \\
	\end{array}
	\right),\left(
	\begin{array}{cccc}
		-\lambda  & 0 & 0 & 0 \\
		0 & -\lambda  & 0 & 0 \\
		0 & a & -b & \lambda -b \\
		0 & \frac{a b}{\lambda -b} & b & b-\lambda  \\
	\end{array}
	\right),
	$$
	$$
	\left(
	\begin{array}{cccc}
		0 & 0 & 0 & 0 \\
		0 & 0 & 0 & 0 \\
		0 & a & b-\lambda  & \lambda -b \\
		0 & -\frac{a b}{\lambda -b} & b & -b \\
	\end{array}
	\right),\left(
	\begin{array}{cccc}
		0 & 0 & 0 & 0 \\
		0 & 0 & 0 & 0 \\
		0 & a & -b-\lambda  & -b-\lambda  \\
		0 & -\frac{a b}{b+\lambda } & b & b \\
	\end{array}
	\right),
	$$
	$$
	\left(
	\begin{array}{cccc}
		-\lambda  & -\lambda  & 0 & 0 \\
		0 & 0 & 0 & 0 \\
		0 & a & b-\lambda  & \lambda -b \\
		0 & -\frac{a b}{\lambda -b} & b & -b \\
	\end{array}
	\right),\left(
	\begin{array}{cccc}
		-\lambda  & \lambda  & 0 & 0 \\
		0 & 0 & 0 & 0 \\
		0 & a & b-\lambda  & \lambda -b \\
		0 & -\frac{a b}{\lambda -b} & b & -b \\
	\end{array}
	\right),$$$$\left(
	\begin{array}{cccc}
		-\lambda  & -\lambda  & 0 & 0 \\
		0 & 0 & 0 & 0 \\
		0 & a & -b-\lambda  & -b-\lambda  \\
		0 & -\frac{a b}{b+\lambda } & b & b \\
	\end{array}
	\right),\left(
	\begin{array}{cccc}
		-\lambda  & \lambda  & 0 & 0 \\
		0 & 0 & 0 & 0 \\
		0 & a & -b-\lambda  & -b-\lambda  \\
		0 & -\frac{a b}{b+\lambda } & b & b \\
	\end{array}
	\right),
	$$
$$
\left(
\begin{array}{cccc}
-\frac{3 \lambda}{2}  & \frac{\lambda }{2} & 0 & 0 \\
	-\frac{\lambda }{2} & -\frac{\lambda }{2} & 0 & 0 \\
	a & a & 0 & 0 \\
	b & b & 0 & 0 \\
\end{array}
\right),\left(
\begin{array}{cccc}
	\frac{\lambda }{2} & \frac{\lambda }{2} & 0 & 0 \\
	-\frac{\lambda }{2} & -\frac{\lambda }{2} & 0 & 0 \\
	a & a & 0 & 0 \\
	b & b & 0 & 0 \\
\end{array}
\right),\left(
\begin{array}{cccc}
-\frac{3 \lambda}{2}  & \frac{\lambda }{2} & 0 & 0 \\
	-\frac{\lambda }{2} & -\frac{\lambda }{2} & 0 & 0 \\
	a & -a & -\lambda  & 0 \\
	b & -b & 0 & -\lambda  \\
\end{array}
\right),\left(
\begin{array}{cccc}
	\frac{\lambda }{2} & \frac{\lambda }{2} & 0 & 0 \\
	-\frac{\lambda }{2} & -\frac{\lambda }{2} & 0 & 0 \\
	a & -a & -\lambda  & 0 \\
	b & -b & 0 & -\lambda  \\
\end{array}
\right),
$$
$$
\left(
\begin{array}{cccc}
	-\frac{\lambda }{2} & -\frac{\lambda }{2} & 0 & 0 \\
	-\frac{\lambda }{2} & -\frac{\lambda }{2} & 0 & 0 \\
	a & -b & -\frac{\lambda }{2} & -\frac{\lambda }{2} \\
	b & -a & -\frac{\lambda }{2} & -\frac{\lambda }{2} \\
\end{array}
\right),\left(
\begin{array}{cccc}
	-\frac{\lambda }{2} & -\frac{\lambda }{2} & 0 & 0 \\
	-\frac{\lambda }{2} & -\frac{\lambda }{2} & 0 & 0 \\
	a & b & -\frac{\lambda }{2} & \frac{\lambda }{2} \\
	b & a & \frac{\lambda }{2} & -\frac{\lambda }{2} \\
\end{array}
\right),
\left(
\begin{array}{cccc}
-\frac{3 \lambda}{2}  & \frac{\lambda }{2} & 0 & 0 \\
	-\frac{\lambda }{2} & -\frac{\lambda }{2} & 0 & 0 \\
	a & -b & -\frac{\lambda }{2} & -\frac{\lambda }{2} \\
	b & -a & -\frac{\lambda }{2} & -\frac{\lambda }{2} \\
\end{array}
\right),$$$$\left(
\begin{array}{cccc}
	\frac{\lambda }{2} & \frac{\lambda }{2} & 0 & 0 \\
	-\frac{\lambda }{2} & -\frac{\lambda }{2} & 0 & 0 \\
	a & -b & -\frac{\lambda }{2} & -\frac{\lambda }{2} \\
	b & -a & -\frac{\lambda }{2} & -\frac{\lambda }{2} \\
\end{array}
\right),\left(
\begin{array}{cccc}
-\frac{3 \lambda}{2}  & \frac{\lambda }{2} & 0 & 0 \\
	-\frac{\lambda }{2} & -\frac{\lambda }{2} & 0 & 0 \\
	a & b & -\frac{\lambda }{2} & \frac{\lambda }{2} \\
	b & a & \frac{\lambda }{2} & -\frac{\lambda }{2} \\
\end{array}
\right),\left(
\begin{array}{cccc}
	\frac{\lambda }{2} & \frac{\lambda }{2} & 0 & 0 \\
	-\frac{\lambda }{2} & -\frac{\lambda }{2} & 0 & 0 \\
	a & b & -\frac{\lambda }{2} & \frac{\lambda }{2} \\
	b & a & \frac{\lambda }{2} & -\frac{\lambda }{2} \\
\end{array}
\right),
$$

$$
\left(
\begin{array}{cccc}
-\frac{3 \lambda}{2}  & -\frac{\lambda }{2} & 0 & 0 \\
	\frac{\lambda }{2} & -\frac{\lambda }{2} & 0 & 0 \\
	a & -a & 0 & 0 \\
	b & -b & 0 & 0 \\
\end{array}
\right),\left(
\begin{array}{cccc}
	\frac{\lambda }{2} & -\frac{\lambda }{2} & 0 & 0 \\
	\frac{\lambda }{2} & -\frac{\lambda }{2} & 0 & 0 \\
	a & -a & 0 & 0 \\
	b & -b & 0 & 0 \\
\end{array}
\right),
\left(
\begin{array}{cccc}
-\frac{3 \lambda}{2}  & -\frac{\lambda }{2} & 0 & 0 \\
	\frac{\lambda }{2} & -\frac{\lambda }{2} & 0 & 0 \\
	a & a & -\lambda  & 0 \\
	b & b & 0 & -\lambda  \\
\end{array}
\right),$$$$\left(
\begin{array}{cccc}
	\frac{\lambda }{2} & -\frac{\lambda }{2} & 0 & 0 \\
	\frac{\lambda }{2} & -\frac{\lambda }{2} & 0 & 0 \\
	a & a & -\lambda  & 0 \\
	b & b & 0 & -\lambda  \\
\end{array}
\right),
\left(
\begin{array}{cccc}
	-\frac{\lambda }{2} & \frac{\lambda }{2} & 0 & 0 \\
	\frac{\lambda }{2} & -\frac{\lambda }{2} & 0 & 0 \\
	a & b & -\frac{\lambda }{2} & -\frac{\lambda }{2} \\
	b & a & -\frac{\lambda }{2} & -\frac{\lambda }{2} \\
\end{array}
\right),\left(
\begin{array}{cccc}
	-\frac{\lambda }{2} & \frac{\lambda }{2} & 0 & 0 \\
	\frac{\lambda }{2} & -\frac{\lambda }{2} & 0 & 0 \\
	a & -b & -\frac{\lambda }{2} & \frac{\lambda }{2} \\
	b & -a & \frac{\lambda }{2} & -\frac{\lambda }{2} \\
\end{array}
\right),
$$
$$
\left(
\begin{array}{cccc}
-\frac{3 \lambda}{2}  & -\frac{\lambda }{2} & 0 & 0 \\
	\frac{\lambda }{2} & -\frac{\lambda }{2} & 0 & 0 \\
	a & b & -\frac{\lambda }{2} & -\frac{\lambda }{2} \\
	b & a & -\frac{\lambda }{2} & -\frac{\lambda }{2} \\
\end{array}
\right),\left(
\begin{array}{cccc}
	\frac{\lambda }{2} & -\frac{\lambda }{2} & 0 & 0 \\
	\frac{\lambda }{2} & -\frac{\lambda }{2} & 0 & 0 \\
	a & b & -\frac{\lambda }{2} & -\frac{\lambda }{2} \\
	b & a & -\frac{\lambda }{2} & -\frac{\lambda }{2} \\
\end{array}
\right),\left(
\begin{array}{cccc}
-\frac{3 \lambda}{2}  & -\frac{\lambda }{2} & 0 & 0 \\
	\frac{\lambda }{2} & -\frac{\lambda }{2} & 0 & 0 \\
	a & -b & -\frac{\lambda }{2} & \frac{\lambda }{2} \\
	b & -a & \frac{\lambda }{2} & -\frac{\lambda }{2} \\
\end{array}
\right),\left(
\begin{array}{cccc}
	\frac{\lambda }{2} & -\frac{\lambda }{2} & 0 & 0 \\
	\frac{\lambda }{2} & -\frac{\lambda }{2} & 0 & 0 \\
	a & -b & -\frac{\lambda }{2} & \frac{\lambda }{2} \\
	b & -a & \frac{\lambda }{2} & -\frac{\lambda }{2} \\
\end{array}
\right),
$$
$$
\left(
\begin{array}{cccc}
-\frac{3 \lambda}{2}  & \frac{\lambda }{2} & 0 & 0 \\
	-\frac{\lambda }{2} & -\frac{\lambda }{2} & 0 & 0 \\
	a & a & 0 & 0 \\
	b & b & 0 & 0 \\
\end{array}
\right),\left(
\begin{array}{cccc}
	\frac{\lambda }{2} & \frac{\lambda }{2} & 0 & 0 \\
	-\frac{\lambda }{2} & -\frac{\lambda }{2} & 0 & 0 \\
	a & a & 0 & 0 \\
	b & b & 0 & 0 \\
\end{array}
\right),\left(
\begin{array}{cccc}
-\frac{3 \lambda}{2}  & \frac{\lambda }{2} & 0 & 0 \\
	-\frac{\lambda }{2} & -\frac{\lambda }{2} & 0 & 0 \\
	a & -a & -\lambda  & 0 \\
	b & -b & 0 & -\lambda  \\
\end{array}
\right),\left(
\begin{array}{cccc}
	\frac{\lambda }{2} & \frac{\lambda }{2} & 0 & 0 \\
	-\frac{\lambda }{2} & -\frac{\lambda }{2} & 0 & 0 \\
	a & -a & -\lambda  & 0 \\
	b & -b & 0 & -\lambda  \\
\end{array}
\right),
$$
$$
\left(
\begin{array}{cccc}
	-\frac{\lambda }{2} & -\frac{\lambda }{2} & 0 & 0 \\
	-\frac{\lambda }{2} & -\frac{\lambda }{2} & 0 & 0 \\
	a & a & 0 & 0 \\
	b & b & 0 & 0 \\
\end{array}
\right),\left(
\begin{array}{cccc}
	-\frac{\lambda }{2} & -\frac{\lambda }{2} & 0 & 0 \\
	-\frac{\lambda }{2} & -\frac{\lambda }{2} & 0 & 0 \\
	a & -a & -\lambda  & 0 \\
	b & -b & 0 & -\lambda  \\
\end{array}
\right),
\left(
\begin{array}{cccc}
	-\frac{\lambda }{2} & \frac{\lambda }{2} & 0 & 0 \\
	\frac{\lambda }{2} & -\frac{\lambda }{2} & 0 & 0 \\
	a & -a & 0 & 0 \\
	b & -b & 0 & 0 \\
\end{array}
\right),\left(
\begin{array}{cccc}
	-\frac{\lambda }{2} & \frac{\lambda }{2} & 0 & 0 \\
	\frac{\lambda }{2} & -\frac{\lambda }{2} & 0 & 0 \\
	a & a & -\lambda  & 0 \\
	b & b & 0 & -\lambda  \\
\end{array}
\right),
$$
where $a,b,c$ are parameters.

\end{theorem}

\section{Relationships with the results of Ma et al. on the Sweedler algebra }

In this section, we shall discuss the relationships between our results and the results of Ma et al. on the Sweedler algebra. Firstly, we recall the Sweedler algebra. Sweedler  algebra $\mathbb{H}_{4}$ is generated by two elements $g$ and $\nu$  which satisfy
$$
g^{2}=1, \nu^{2}=0, g\nu+\nu g=0.
$$
The comultiplication, the antipode and the counit of  $\mathbb{H}_{4}$ are given by
$$
\Delta(g)=g\otimes g, \Delta(\nu)=g\otimes\nu+\nu\otimes 1, \varepsilon(g)=1, \varepsilon(\nu)=0, S(g)=g, S(\nu)=-g\nu.
$$
Notice that the dimension of $\mathbb{H}_{4}$ is four, and $1,g,\nu, g\nu$ form a basis for  $\mathbb{H}_{4}$. By taking $e_{1}=g,e_{2}=\nu, e_{3}=gv$, we observe that  $\mathbb{H}_{4}$ as an algebra  is just a split semi-quaternion algebra.

In \cite{Ma}(also see  \cite{Ma1}), some  Rota-Baxter operators on $\mathbb{H}_{4}$ are described. It was proven that the following  operators are  Rota-Baxter operators:
\begin{itemize}
	\item [(S1)]~$\mathcal{P}(1)=0, \mathcal{P}(g)=0, \mathcal{P}(\nu)=0, \mathcal{P}(g\nu)=0$,
	\item [(S2)]~$\mathcal{P}(1)=0, \mathcal{P}(g)=0, \mathcal{P}(\nu)=-\lambda \nu, \mathcal{P}(g\nu)=-\lambda g\nu$,
	\item [(S3)]~$\mathcal{P}(1)=-\lambda 1, \mathcal{P}(g)=-\lambda g, \mathcal{P}(\nu)=0, \mathcal{P}(g\nu)=0$,
	\item [(S4)]~$\mathcal{P}(1)=-\lambda 1, \mathcal{P}(g)=-\lambda g, \mathcal{P}(\nu)=-\lambda \nu, \mathcal{P}(g\nu)=-\lambda g\nu$,
	\item [(S5)]~$\mathcal{P}(1)=0, \mathcal{P}(g)=-a 1+a g-\frac{(\lambda + a)(\lambda+a+b)}{c}\nu+\frac{(\lambda + a)(\lambda+b)}{c}g\nu, \mathcal{P}(\nu)=-c 1+c g-(2\lambda+a+b)\nu+(\lambda+b)g\nu, \mathcal{P}(g\nu)=-c 1+c g-(\lambda+a+b)\nu+bg\nu$,
	\item [(S6)]~$\mathcal{P}(1)=-\lambda 1, \mathcal{P}(g)=(\lambda+a) 1+a g-\frac{(\lambda + a)(\lambda+a+b)}{c}\nu+\frac{(\lambda + a)(\lambda+b)}{c}g\nu, \mathcal{P}(\nu)=c 1+c g-(2\lambda+a+b)\nu+(\lambda+b)g\nu, \mathcal{P}(g\nu)=c 1+c g-(\lambda+a+b)\nu+bg\nu$,
	\item [(S7)]~$\mathcal{P}(1)=-\lambda 1, \mathcal{P}(g)=\lambda1+a \nu+\frac{ab}{\lambda+b}g\nu, \mathcal{P}(\nu)=-(\lambda+b)\nu-bg\nu, \mathcal{P}(g\nu)=(\lambda+b)\nu+bg\nu$,
	\item [(S8)]~$\mathcal{P}(1)=-\lambda 1, \mathcal{P}(g)=\lambda 1+\frac{\lambda(\lambda + a)}{b}\nu+\frac{\lambda(\lambda + a)}{b}g\nu, \mathcal{P}(\nu)=-b 1-b g-(2\lambda+a)\nu-(\lambda+a)g\nu, \mathcal{P}(g\nu)=b 1+b g+(\lambda+a)\nu+ag\nu$,
	\item [(S9)]~$\mathcal{P}(1)=\frac{1}{2}\lambda 1-\frac{1}{2}\lambda g+a \nu + b g\nu, \mathcal{P}(g)=\frac{1}{2}\lambda 1-\frac{1}{2}\lambda g-b \nu + a g\nu, \mathcal{P}(\nu)=-\frac{1}{2}\lambda \nu - \frac{1}{2}\lambda g\nu, \mathcal{P}(g\nu)=-\frac{1}{2}\lambda \nu - \frac{1}{2}\lambda g\nu$,
\end{itemize}
where $a,b,c$ are parameters.

From (S1)-(S9), we have the following matrices of $\mathcal{P}$ with respect to $1,g,\nu,g\nu$: 
$$
\left(
\begin{array}{cccc}
	0 & 0 & 0 & 0 \\
	0 & 0 & 0 & 0 \\
	0 & 0 & 0 & 0 \\
	0 & 0 & 0 & 0 \\
\end{array}
\right)(1), \left(
\begin{array}{cccc}
	0 & 0 & 0 & 0 \\
	0 & 0 & 0 & 0 \\
	0 & 0 & -\lambda  & 0 \\
	0 & 0 & 0 & -\lambda  \\
\end{array}
\right)(2),\left(
\begin{array}{cccc}
	-\lambda  & 0 & 0 & 0 \\
	0 & -\lambda  & 0 & 0 \\
	0 & 0 & 0 & 0 \\
	0 & 0 & 0 & 0 \\
\end{array}
\right)(3),\left(
\begin{array}{cccc}
	-\lambda  & 0 & 0 & 0 \\
	0 & -\lambda  & 0 & 0 \\
	0 & 0 & -\lambda  & 0 \\
	0 & 0 & 0 & -\lambda  \\
\end{array}
\right)(4),
$$
$$
\left(
\begin{array}{cccc}
	0 & -a & -c & -c \\
	0 & a & c & c \\
	0 & -\frac{(a+\lambda ) (a+b+\lambda )}{c} & -(a+b+2 \lambda ) & -(a+b+\lambda ) \\
	0 & \frac{(a+\lambda ) (b+\lambda )}{c} & b+\lambda  & b \\
\end{array}
\right)(5),
$$
$$
\left(
\begin{array}{cccc}
	-\lambda  & a+\lambda  & c & c \\
	0 & a & c & c \\
	0 & -\frac{(a+\lambda ) (a+b+\lambda )}{c} & -(a+b+2 \lambda ) & -(a+b+\lambda ) \\
	0 & \frac{(a+\lambda ) (b+\lambda )}{c} & b+\lambda  & b \\
\end{array}
\right)(6),
$$
$$
\left(
\begin{array}{cccc}
	-\lambda  & \lambda  & 0 & 0 \\
	0 & 0 & 0 & 0 \\
	0 & a & -(b+\lambda ) & b+\lambda  \\
	0 & \frac{a b}{b+\lambda } & -b & b \\
\end{array}
\right)(7),\left(
\begin{array}{cccc}
	-\lambda  & \lambda  & -b & b \\
	0 & 0 & -b & b \\
	0 & \frac{\lambda  (a+\lambda )}{b} & -(a+2 \lambda ) & a+\lambda  \\
	0 & \frac{\lambda  (a+\lambda )}{b} & -(a+\lambda ) & a \\
\end{array}
\right)(8),
$$
$$
\left(
\begin{array}{cccc}
	\frac{\lambda }{2} & \frac{\lambda }{2} & 0 & 0 \\
	-\frac{\lambda }{2} & -\frac{\lambda }{2} & 0 & 0 \\
	a & -b & -\frac{\lambda }{2} & -\frac{\lambda }{2} \\
	b & a & -\frac{\lambda }{2} & -\frac{\lambda }{2} \\
\end{array}
\right)(9).
$$
Observe easily that the matrices (1),(2),(3),(4),(7),(9) just the special cases of Theorem \ref{k1}. For the matrix $$
\left(
\begin{array}{cccc}
	0 & -b & a & a \\
	0 & b & -a & -a \\
	0 & \frac{(b+c) (b+\lambda )}{a} & -b-c-\lambda  & -b-c \\
	0 & -\frac{c (b+\lambda )}{a} & c & c-\lambda  \\
\end{array}
\right)$$in Theorem \ref{k1}, by replacing  $b, a, c$ with $a,-c,\lambda+b$ in the matrix above, we have the desired matrix (5). Similarly, we can gain the desired (6) and (8). So we have checked that the matrices (1)-(9) are special cases of Theorem \ref{k1}, which means that the operators (S1)-(S9) are exact Rota-Baxter operators.

\section*{summary}
In this paper, we describe all the Rota-Baxter operators with any weight on split semi-quaternion algebra. By applying to Sweedler algebra, we gain all  the Rota-Baxter operators with any weight on Sweedler algebra. Theorems \ref{k2} and \ref{k1} are very important. In the future work, we want to use  Theorems \ref{k2} and \ref{k1} to discuss the Rota-Baxter operators on other types of quaternion algebra.

\end{document}